\documentclass[11pt]{amsart}
\usepackage{amsmath} \usepackage{amsxtra}
\usepackage{amscd} \usepackage{amsthm}  
\usepackage{amsfonts}\usepackage{amssymb} 
\usepackage{eucal}\usepackage{bm}
\usepackage{graphicx} 
\newtheorem{Lm}[subsubsection]{Lemma}
\newtheorem{Pp}[subsubsection]{Proposition}

\newtheorem{Thm}[subsubsection]{Theorem}
\newtheorem{Def}[subsubsection]{Definition}
\newtheorem{Rem}[subsubsection]{Remark}

\theoremstyle{definition}

\theoremstyle{remark}
\newcommand{\nc}{\newcommand}
\nc{\renc}{\renewcommand}
\nc{\ssec}{\subsection}
\nc{\sssec}{\subsubsection}
\nc{\on}{\operatorname}

\newcommand{\cA}{{\mathcal A}}

\newcommand{\cH}{{\mathcal H}}

\newcommand{\cG}{{\mathcal G}}

\newcommand{\cJ}{{\mathcal J}}

\newcommand{\cL}{{\mathcal L}}

\newcommand{\cF}{{\mathcal F}}

\newcommand{\NN}{{\mathbb N}}
\newcommand{\ZZ}{{\mathbb Z}}
\newcommand{\QQ}{{\mathbb Q}}
\newcommand{\CC}{{\mathbb C}}

\nc{\gJ}{\mathfrak{J}}
\nc{\gL}{\mathfrak{L}}

\newcommand{\Rep}{{\on{Rep}}}

\newcommand{\Qlb}{\mathbb{\bar Q}_\ell}

\newcommand{\Ql}{\mathbb{Q}_\ell}
\newcommand{\toup}[1]{\stackrel{#1}{\to}}
\newcommand{\hook}[1]{\stackrel{#1}{\hookrightarrow}}

\newcommand{\Sp}{\on{\mathbb{S}p}}

\newcommand{\Hom}{\on{Hom}}

\newcommand{\Aut}{\on{Aut}}

\newcommand{\Spec}{\on{Spec}}

\newcommand{\Fr}{{\on{Fr}}}

\newcommand{\id}{\on{id}}
\newcommand{\tr}{\on{tr}}
\newcommand{\QED}{$\square$} 
  
\newcommand{\Fp}{\mathbb{F}_p}  % what for??????    
\newcommand{\bFp}{\ov{\mathbb{F}}_p}

\newcommand{\iso}{{\widetilde\to}}

\newcommand{\comp}{\circ}

\newcommand{\wt}{\widetilde}
\newcommand{\ov}[1]{\overline{#1}}
\newcommand{\select}[1]{{\it{#1}}}

\newcommand{\Hecke}{\on{Hecke}}

\newcommand{\Vect}{\on{Vect}}

\nc{\Hi}{\on{Hi}}
\nc{\Lo}{\on{Lo}}
\nc{\Inv}{\on{Inv}}
\nc{\Fix}{\on{Fix}}
\nc{\oV}{\overset{\scriptscriptstyle\circ}{V}}
\nc{\ssharp}{{\scriptstyle\sharp}}
\nc{\oHecke}{\overset{\scriptstyle\bullet}{\Hecke}}
\nc{\Ran}{\on{Ran}}
\nc{\Whit}{\on{Whit}}
\nc{\Maps}{\on{{\mathcal M}aps}}
\nc{\DGCat}{\on{DGCat}}
\nc{\colim}{\on{colim}}
\nc{\PreStk}{\on{PreStk}}
\nc{\otimesshriek}{\stackrel{!}{\otimes}}
\nc{\Cat}{\on{Cat}}
\nc{\QCoh}{\on{QCoh}}
\nc{\FactMod}{\on{FactMod}}
\nc{\oo}[1]{\overset{\scriptscriptstyle\circ}{#1}}
\nc{\inv}{\on{inv}}
\nc{\IndCoh}{\on{IndCoh}}
\nc{\Groth}{\on{Groth}}
\nc{\coGroth}{\on{coGroth}}
\nc{\Conf}{\on{Conf}}
\nc{\Map}{\on{Map}}
\nc{\Tr}{\on{Tr}}
\nc{\oblv}{\on{oblv}}
\nc{\Fact}{\on{Fact}}
\nc{\e}{\mathrm{e}}
\nc{\DG}{\on{DG}}
\nc{\Spc}{\on{Spc}}
\nc{\bfitDelta}{\bm{\mathit{\Delta}}}
\nc{\Grp}{\on{{\mathcal G}rp}}
\nc{\Fib}{\on{Fib}}
\nc{\cofib}{\on{cofib}}
\nc{\PsId}{\on{Ps-Id}}
\nc{\Sptr}{\on{Sptr}}
\nc{\Reg}{\on{Reg}}
\nc{\bb}[1]{\overset{\scriptscriptstyle\bullet}{#1}}
\nc{\ind}{\on{ind}}
\nc{\FactGe}{\on{FactGe}}
\nc{\SI}{\on{SI}}
\makeatletter
\newcommand*{\doublerightarrow}[2]{\mathrel{
  \settowidth{\@tempdima}{$\scriptstyle#1$}
  \settowidth{\@tempdimb}{$\scriptstyle#2$}
  \ifdim\@tempdimb>\@tempdima \@tempdima=\@tempdimb\fi
  \mathop{\vcenter{
    \offinterlineskip\ialign{\hbox to\dimexpr\@tempdima+1em{##}\cr
    \rightarrowfill\cr\noalign{\kern.5ex}
    \rightarrowfill\cr}}}\limits^{\!#1}_{\!#2}}}
\newcommand*{\triplerightarrow}[1]{\mathrel{
  \settowidth{\@tempdima}{$\scriptstyle#1$}
  \mathop{\vcenter{
    \offinterlineskip\ialign{\hbox to\dimexpr\@tempdima+1em{##}\cr
    \rightarrowfill\cr\noalign{\kern.5ex}
    \rightarrowfill\cr\noalign{\kern.5ex}
    \rightarrowfill\cr}}}\limits^{\!#1}}}
\makeatother

\begin{document}
\title{Towards canonical representations of finite Heisenberg groups}
\author{S. Lysenko}
\address{Institut \'Elie Cartan Nancy (Math\'ematiques), Universit\'e de Lorraine, B.P. 239, F-54506 Vandoeuvre-l\`es-Nancy Cedex, France}
\email{Sergey.Lysenko@univ-lorraine.fr}
\begin{abstract} 
We consider a finite abelian group $M$ of odd exponent $n$ with a symplectic form $\omega: M\times M\to \mu_n$ and the Heisenberg extension $1\to \mu_n\to H\to M\to 1$ with the commutator $\omega$. 
According to the Stone - von Neumann theorem, $H$ admits an irreducible representation with the tautological central character (defined up to a non-unique isomorphism). We construct such irreducible representation of $H$ defined up to a unique isomorphism, so canonical in this sense.
\end{abstract}
\maketitle 
\section{Introduction} 

\sssec{} 
\label{Sect_1.0.1}
Consider a finite abelian group $M$ of odd exponent $n$ with a symplectic form $\omega: M\times M\to \mu_n$. It admits a unique symmetric Heisenberg extension $1\to \mu_n\to H\to M\to 1$ with the commutator $\omega$. According to the Stone - von Neumann theorem, $H$ admits an irreducible representation with the tautological central character (defined up to a non-unique isomorphism). We construct such irreducible representation of $H$ defined up to a unique isomorphism, so canonical in this sense, over a suitable finite extension of $\QQ$.

\sssec{} We are motivated by the following question of Dennis Gaitsgory about \cite{L}. Let $X$ be a smooth projective connected curve over an algebraically closed field $k$. Let $T$ be a torus over $k$ with a geometric metaplectic data $\cG$ as in \cite{GL}. To fix ideas, consider the sheaf-theoretic context of $\ell$-adic sheaves on finite type schemes over $k$. Let $(H, \cG_{Z_H},\epsilon)$ be the metaplectic Langlands dual datum associated to $(T,\cG)$ in (\cite{GL}, Section~6), so $H$ is a torus over $\Qlb$ isogenous to the Langlands dual of $T$. Let $\sigma$ be a twisted local system on $X$ for $(H, \cG_{Z_H})$ in the sense of (\cite{GL}, Section~8.4). To this data in (\cite{GL}, Section~9.5.3) we attached the $\DG$-category of Hecke eigensheaves. The question is whether this category identifies canonically with the $\DG$-category $\Vect$ of $\Qlb$-vector spaces. 
In \cite{L} we constructed such an irreducible Hecke eigensheaf for $\sigma$ out of a given irreducible representation of certain finite Heisenberg group (denoted by $\Gamma$ given by formula (33) in \cite{L}, Section~5.2.4) with the tautological central character.

 This is why we are interested in constructing a canonical irreducible representation of finite Heisenberg groups as in Section~\ref{Sect_1.0.1}. We do this only assuming the order of $M$  odd, the case of even order remains open. 

\section{Main result}

\sssec{} Let $e$ be  an algebraically closed field of characteristic zero. 
Let $M$ be a finite abelian group, $\omega: M\times M\to e^*$ a bilinear form, which is alternating, that is, $\omega(m,m)=0$ for any $m\in M$. Assume the induced map $M\to \Hom(M, e^*)$ is an isomorphism, that is, the form is nondegenerate. 
 
 If $L\subset M$ is a subgroup, $L^{\perp}=\{m\in M\mid \omega(m, l)=0\; \mbox{for all}\; l\in L\}$ is its orthogonal complement, this is a subgroup. The group $L$ is isotropic if $L\subset L^{\perp}$. The subgroup $L$ is lagrangian if $L^{\perp}=L$. For a lagrangian subgroup we get an exact sequence 
\begin{equation}
\label{seq_1} 
0\to L\to M\to L^*\to 0,
\end{equation}
where $L^*=\Hom(L, e^*)$. Namely, we send $m\in M$ to the character $l\mapsto \omega(m, l)$ of $L$. This exact sequence always admits a splitting $L^*\to M$, which is a homomorphism, see for example  (\cite{P}, 4.1). For such a splitting after the obtained identification $M\,\iso\, L\times L^*$ the form $\omega$ becomes 
\begin{equation}
\label{eq_1}
\omega((l_1, \chi_1), (l_2,\chi_2))=\frac{\chi_1(l_2)}{\chi_2(l_1)}
\end{equation}
for $l_i\in L, \chi_i\in L^*$. 

% Let $d>0$ we be such that $d^2=\mid M\mid$. 
 
 By (\cite{P}, Theorem~2), up to an isomorphism, there is a unique central extension 
\begin{equation}
\label{ext_H_e^*}
1\to e^*\to H_{e^*}\to M\to 1
\end{equation}
with the commutator $\omega$. We are interested in understanding the category of representations of $H_{e^*}$ with the tautological central character.

\sssec{} 
\label{Sect_2.0.2} For a finite abelian group $L$ its exponent is the least common multiple of the orders of the elements of $L$. Let $n$ be the exponent of $M$, this is a divisor of $\sqrt{\mid M\mid}\in\NN$. 
 
 Let $\mu_n=\mu_n(e)$. Let us be given a central extension 
\begin{equation}
\label{def_H_ext}
1\to \mu_n\to H\to M\to 1
\end{equation}
together with a symmetric structure $\sigma$ in the sense of (\cite{B}, Section~1.1) and commutator $\omega$. That is, $\sigma$ is an automorphism of $H$ such that $\sigma^2=\id$, $\sigma\mid_{\mu_b}=\id$, and $\sigma\mod \mu_b$ is the involution $m\mapsto -m$ of $M$. 

 From now on assume $n$ odd. Then by (\cite{P}, Section~1), there is a unique symmetric central extension (\ref{def_H_ext}) up to a unique isomorphism. Besides, (\ref{ext_H_e^*}) is isomorphic to the push-out of (\ref{def_H_ext}) under the tautological character $\iota: \mu_n\hook{} e^*$.
 
  The extension $H$ is constructed as follows. Let $\beta: M\times M\to \mu_n$ be the unique alternating bilinear form such that $\beta^2=\omega$. We take $H=M\times\mu_n$ with the product 
$$
(m_1, a_1)(m_2, a_2)=(m_1+m_2, a_1a_2\beta(m_1, m_2))
$$
for $m_i\in M, a_i\in \mu_n$. Then $\sigma(m,a)=(-m, a)$ for $m\in M, a\in\mu_n$.   

 Let $G=\Sp(M)$, the group of automorphisms of $M$ preserving $\omega$. Let $g\in G$ act on $H$ sending $(m,a)$ to $(gm, a)$. This gives the semi-direct product $H\rtimes G$. 

\sssec{}  The following version of the Stone - von Neumann theorem holds for $H$, the proof is left to a reader. 
\begin{Pp} Up to an isomorphism, there is a unique irreducible representation of $H$ over $e$ with the tautological central character $\iota: \mu_d\hook{} e^*$.
\end{Pp}

\sssec{} 
\label{Sect_2.0.4} Write $\cL(M)$ for the set of lagragian subgroups in $M$. For $L\in \cL(M)$ let $\bar L$ be the preimage of $L$ in $H$, this is a subgroup. If $\chi_L: \bar L\to e^*$ is a character extending the  
tautological character $\iota: \mu_b\hook{} e^*$, set
$$
\cH_L=\{f: H\to e\mid f(\bar l h)=\chi_L(\bar l)f(h),\; \mbox{for}\; \bar l\in\bar L, h\in H\}
$$
This is a representation of $H$ by right translations. It is irreducible with  central character $\iota$. 

\sssec{} 
\label{Sect_2.0.5}
We study the following.%\footnote{to answer a question of D. Gaitsgory about \cite{L}.}   

\medskip\noindent
{\bf Problem}: Describe the category $\Rep_{\iota}(H)$ of representations of $H$ over $e$ with central character $\iota: \mu_b\hook{} e^*$. Is there an object of $\Rep_{\iota}(H)$, which is irreducible and defined up to a unique isomorphism? (If yes, it would provide an equivalence between $\Rep_{\iota}(H)$ and the category of $e$-vector spaces).  
 
\sssec{} Let $I$ be the set of primes appearing in the decomposition of $n$, write $n=\prod_{p\in I} p^{r(p)}$ with $r(p)>0$. Let $K\subset e$ be the subfield generated over $\QQ$ by $\{\sqrt{p}\mid p\in I\}$ and $\mu_n$. 
\begin{Thm} 
\label{Th_1} There is an irreducible representation $\pi$ of $H$ over $K$ with central character $\iota: \mu_n\hook{} K^*$ defined up to a unique isomorphism. The $H$-action on $\pi$ extends naturally to an action of $H\rtimes G$. 
\end{Thm}

\begin{Rem} Let $K'\subset e$ be the subfield generated over $\QQ$ by $\mu_n$. The field of definition of the character of $\pi$ is $K'$.
However, we do not expect that for any $H$ with $n$ odd Theorem~\ref{Th_1} holds already with $K$ replaced by $K'$, but we have not checked that. Our choices of $\sqrt{p}$ for $p\in I$ are made to use the results of \cite{LL}, and the formulas from \cite{LL} do not work without these choices. Note also the following. If $L$ is an odd abelian group, and $b: L\times L\to e^*$ is a nondegenerate symmetric bilinear form then the Gauss sum of $b$ is defined as 
$$
G(L,b)=\sum_{l\in L} b(l,l)
$$
Using the classification of such symmetric bilinear forms given in \cite{Ko}, one checks that  $G(L,b)^4=\mid L\mid^2$. Since the construction of $\pi$ in Theorem~\ref{Th_1} is related to representing the corresponding 2-cocyle (given essentially by certain Gauss sums) as a coboundary (after some minimal additional choices), we expect that our choices of $\sqrt{p}$ for $p\in I$ are necessary.
\end{Rem}

\begin{Rem} For $L\in\cL(M)$, the $H$-representation $\cH_L$ from Section~\ref{Sect_2.0.4} is defined over $K$. We sometimes view it as a representation over $K$, the precise meaning is hopefully clear from the context.
\end{Rem}

\section{Proof of Theorem~\ref{Th_1}}
\label{Sect_3}

\sssec{Reduction} For $p\in I$ let 
$$
H_p=\{h\in H\mid h^{(p^s)}=1\;\mbox{for}\; s\; \mbox{large enough}\}
$$
and similarly for $M_p$. Then $H_p\subset H$ is a subgroup that fits into an exact sequence $1\to \mu_{p^{r(p)}}\to H_p\to M_p\to 1$, and $H=\prod_{p\in I} H_p$, product of groups. Indeed, $\omega(H_p, H_q)=1$ for $p, q\in I$, $p\ne q$. Besides, $\sigma$ preserves $H_p$ for each $p\in I$, so $(H_p, \sigma\!\mid_{H_p})$ is a symmetric Heisenberg extension of $(M_p,\omega_p)$ by $\mu_{p^{r(p)}}$. Here $\omega_p: M_p\times M_p\to \mu_{p^{r(p)}}$ is the restriction of $\omega$. So, Problem~\ref{Sect_2.0.5} reduces to the case of a prime $n$. If $M_p=0$ then take $\pi_p$ be the 1-dimensional representation given by the tautological character $\mu_{p^{r(p)}}\hook{} e^*$. 

 For $p\in I$ odd let $K_p\subset e$ be the subfield generated over $\QQ$ by $\mu_{p^{r(p)}}$ and $\sqrt{p}$. We prove Theorem~\ref{Th_1} in the case of an odd prime $n$ getting for $p\in I$ a representation $\pi_p$ of $H_p$ over $K_p$, hence over $K$ also. Then for any odd $n$, $\pi=\otimes_{p\in I} \pi_p$ is the desired representation. 

\sssec{} From now on we assume $n=p^r$ for an odd prime $p$. 

\ssec{Case $r=1$} 
\label{Sect_3.1}

\sssec{} In this section we assume $M$ is a $\Fp$-vector space of dimension $2d$. To apply the results of \cite{LL}, pick an isomorphism $\psi: \Fp\to\mu_p$. It allows to identify $H$ with $M\times\Fp$. We then   
view $\cL(M), H$ as algebraic varieties over $\Fp$. We allow the case $d=0$ also.

\sssec{} Recall the following construction from (\cite{LL}, Theorem~1).\footnote{For this construction we adopt the conventions of \select{loc.cit} about $\ZZ/2\ZZ$-gradings and \'etale $\Qlb$-sheaves on schemes over $\Fp$.} 

 Pick a prime $\ell\ne p$, and an algebraic closure $\Qlb$ of $\Ql$. We assume $\Qlb$ is chosen in such a way that $K\subset \Qlb$ is a subfield. In particular, we get $\sqrt{p}\in K\subset\Qlb$. It gives rise to the $\Qlb$-sheaf $\Qlb(\frac{1}{2})$ over $\Spec\Fp$.
  
 Pick a 1-dimensional $\Fp$-vector space $\cJ$ of parity $d\mod 2$ as $\ZZ/2\ZZ$-graded. Let $\cA$ be the line bundle (of parity zero as $\ZZ/2\ZZ$-graded) on $\cL(M)$ with fibre $\cJ\otimes \det L$ at $L\in\cL(M)$. Write $\wt\cL(M)$ for the gerbe of square roots of $\cA$.   
 
 In \select{loc.cit} we have constructed an irreducible perverse sheaf $F$ on $\wt\cL(M)\times\wt\cL(M)\times H$. Though in  \select{loc.cit.} we mostly worked over an algebraic closure $\bFp$, $F$ is defined over $\Fp$. 
 
\begin{Lm} For any $i: \Spec\Fp\to \wt\cL(M)\times\wt\cL(M)\times H$, $\tr(\Fr, i^*F)\in K$. Here $\Fr$ is the geometric Frobenius endomorphism.
\end{Lm}
\begin{proof}
This follows from formula (10) in (\cite{LL}, Section~3.3). Namely, after a surjective smooth localization (a choice of an additional lagrangian in $M$),  there is an explicit formula for $F$ as the convolution along $H$ of two explicit rank one local systems. Their traces of Frobenius lie in $K$, as their definition involves only the Artin-Schreier sheaf and Tate twists. So, the same holds after the convolution along the finite group $H(\Fp)$. 
\end{proof}

\sssec{} 
\label{Sect_3.1.4}
For an algebraic stack $S\to\Spec \Fp$ we write $S(\Fp)$ for the set of isomorphism classes of its $\Fp$-points. In view of the isomorphism $\psi:\Fp\,\iso\,\mu_p$ fixed above,  for $L\in\cL(M)(\Fp)$ we identify $\bar L=L\times \mu_p$ with $L\times \Fp$. Let 
$$
F^{cl}: \wt\cL(M)(\Fp)\times\wt\cL(M)(\Fp)\times H(\Fp)\to K
$$ 
be the function trace of Frobenius of $F$. 

 For $L\in\cL(M)(\Fp)$ its preimage in $\wt\cL(M)(\Fp)$ consists of two elements. We let $\mu_2$ act on $\wt\cL(M)(\Fp)$ over $\cL(M)(\Fp)$ permuting the elements in the preimage of each $L\in\cL(M)(\Fp)$. We call a function $h: \wt\cL(M)(\Fp)\to K$ genuine if it changes by the nontrivial character of $\mu_2$ under this $\mu_2$-action. Recall that $F^{cl}$ is genuine with respect the first and the second variable.
 
 Let us write $L^0$ for a point of $\wt\cL(M)(\Fp)$ over $L\in \cL(M)(\Fp)$. As in (\cite{LL}, Section~2), for $L^0, N^0\in \wt\cL(M)(\Fp)$ viewing 
$\cH_L, \cH_N$ as $H$-representations over $K$, we define the canonical intertwining operator 
$$
F_{N^0, L^0}: \cH_L\to \cH_N
$$
by
$$
(F_{N^0, L^0}f)(h_1)=\int_{h_2\in H} F_{N^0, L^0}(h_1h_2^{-1})f(h_2)dh_2,
$$
where our measure $dh_2$ is normalized by requiring that the volume of a point is one. 

 Let $G=\Sp(M)$ viewed as an algebraic group over $\Fp$. It acts naturally on $\cL(M), H$, and $\wt\cL(M)$. 
By definition, for $g\in G, (m,a)\in H$, $g(m,a)=(gm, a)$ for $m\in M, a\in\Fp$, and this action preserves the symmetric structure $\sigma$ on $H$. If $g\in G$, $f: H\to K$ then $gf: H\to K$ is given by $(gf)(h)=f(g^{-1}h)$. Then $g\in G(\Fp)$ yields an isomorphism $\cH_L\,\iso\, \cH_{gL}$. We let $G$ act diagonally on $\wt\cL(M)\times\wt\cL(M)\times H$. 
 
 The above intertwining operators satisfy the following properties

\begin{itemize}
\item $F_{L^0, L^0}=\id$;
\item $F_{R^0, N^0}\comp F_{N^0, L^0}=F_{R^0, L^0}$ for any $R^0, N^0, L^0\in \wt\cL(M)(\Fp)$;
\item for any $g\in G(\Fp)$ we have $g\comp F_{N^0, L^0}\comp g^{-1}=F_{gN^0, gL^0}$. 
\end{itemize}
 
\begin{Def} Let $\pi$ be the $K$-vector space of collections $f_{L^0}\in \cH_L$ for $L^0\in \wt\cL(M)(\Fp)$ satisfying the property: for $N^0, L^0\in \wt\cL(M)(\Fp)$ one has 
$$
F_{N^0, L^0}(f_{L^0})=f_{N^0}
$$
This is our canonical $H$-representation over $K$. 
\end{Def} 

 We let $G(\Fp)$ act on $\wt\cL(M)(\Fp)\times H(\Fp)$ diagonally. This yields a $G(\Fp)$-action on $\pi$ sending $\{f_{L^0}\}\in \pi$ to the collection $L^0\mapsto g(f_{g^{-1}L^0})$. 
 
\ssec{Case $r\ge 1$}

\sssec{} Let $L$ be a finite abelian group, $p$ be any prime number. For $k\ge 0$ let $L[p^k]=\{l\in L\mid p^kl=0\}$ and 
$$
\rho_k(L)=L[p^k]/(L[p^{k-1}]+pL[p^{k+1}])
$$
Each $\rho_k(L)$ is a vector space over $\Fp$. Note that 
$$
\rho_k(\ZZ/p^m\ZZ)\,\iso\, \left\{
\begin{array}{ll}
\ZZ/p\ZZ, & m=k\\
0, & \mbox{otherwise}
\end{array}
\right.
$$
For finite abelian groups $L, L'$ one has canonically $
\rho_k(L\times L')\,\iso\, \rho_k(L)\times \rho_k(L')$.

\sssec{Canonical isotropic subgroup} Let $p$ be any prime, $M$ is a finite abelian $p$-group of exponent $n=p^r$ with an alternating nondegenerate bilinear form $\omega: M\times M\to \mu_n$. We first construct by induction a canonical isotropic subgroup $S\subset M$ such that $\Aut(M)$ fixes $S$ and $S^{\perp}/S$ is a $\Fp$-vector space.

 Write the set $\{r>0\mid \rho_r(M)\ne 0\}$ as $\{r_1,\ldots, r_s\}$ with $0<r_1<r_2<\ldots <r_s$. There is an orthogonal direct sum $(M,\omega)\,\iso\, \oplus_{i=1}^s (M_i,\omega_i)$, where $\omega_i: M_i\times M_i\to \mu_n$ is an alternating nondegenerate bilinear form, and $M_i$ is a free $\ZZ/p^{r_i}$-module of finite rank. 
 
 Let 
$$
r'=\left\{
\begin{array}{ll}
\frac{r_s}{2}, & r_s\; \mbox{is even}\\
\frac{r_s+1}{2}, & r_s\; \mbox{is odd}
\end{array}
\right.
$$
Set $S_1=p^{r'}M$. Since $\omega$ takes values in $\mu_{p^{r_s}}$, $S_1$ is isotropic and fixed by $\Aut(M)$. By induction hypothesis, we have a canonical isotropic subgroup $S'\subset M_1:=S_1^{\perp}/S_1$ such that $S'^{\perp}/S'$ is a $\Fp$-vector space, where $S'^{\perp}$ denotes the orthogonal complement of $S'$ in $M_1$. Let $S$ be the preimage of $S'$ under $S_1^{\perp}\to M_1$. This is our canonical isotropic subgroup in $M$. 

 Set $M_c=S^{\perp}/S$, it is equipped with the induced alternating nondegenerate bilinear form $\omega_c: M_c\times M_c\to \mu_p$, the subscript $c$ stands for `canonical'.  
 
\sssec{} We keep the assumptions of Theorem~\ref{Th_1}, so $p$ is odd. View $S$ as a subgroup of $H$ via $s\mapsto (s,0)\in H$ for $s\in S$. Let $H^S=S^{\perp}\times \mu_n$, this is a subgroup of $H$. Since $S$ lies in the kernel of $\beta:  S^{\perp}\times S^{\perp}\to\mu_n$, we get the alternating nondegenerate bilinear form $\beta_c: M_c\times M_c\to\mu_p$ given by $\beta_c(m_1, m_2)=\beta(\tilde m_1,\tilde m_2)$ for $\tilde m_i\in S^{\perp}$ over $m_i$. 

 Set $H_c=M_c\times \mu_p$ with the product
$$
(m_1, a_1)(m_2, a_2)=(m_1+m_2, a_1a_2\beta_c(m_1 m_2))
$$
This is a central extension $1\to \mu_p\to H_c\to M_c\to 1$ with the commutator $\omega_c$ and the symmetric structure $\sigma_c(m,a)=(-m,a)$ for $(m,a)\in H_c$. 

 Let $\alpha_S: H^S\to H_c$ be the homomorphism sending $(m,a)$ to $(m\!\mod S, a)$ for $m\in S^{\perp}$, its kernel is $S$.  

 As in Section~\ref{Sect_3.1}, we get the algebraic stack $\wt\cL(M_c), \cL(M_c), H_c$ over $\Fp$. Let $G=\Sp(M,\omega)$ be the group of automorphisms of $M$ preserving $\omega$, this is a finite group. We let $g\in G$ act on $H$ sending $(m,a)$ to $(gm, a)$. Let $g\in G$ act on functions $f: H\to K$ by $(gf)(h)=f(g^{-1}h)$ for $h\in H$. For $L\in\cL(M)$ this yields an isomorphism $g: \cH_L\,\iso\, \cH_{gL}$ of $K$-vector spaces. 
 
 Since $G$ preserves $S^{\perp}$, we have the homomorphism $G\to G_c:=\Sp(M_c)(\Fp)$. Via this map, $G$ acts on $\cL(M_c)(\Fp)$, $\wt\cL(M_c)(\Fp)$, $H_c$. 
 
\sssec{} We denote elements of $\cL(M_c)$ by a capital letter with a subscript $c$. For $L_c\in \cL(M_c)$ let $L\in \cL(M)$ denote the preimage of $L_c$ under $S^{\perp}\to M_c$. 

 For $L_c\in \cL(M_c)$ we have the representation $\cH_{L_c}$ of $H_c$ over $K$ defined in Section~\ref{Sect_3.1.4}, and the $H$-representation $\cH_L$ over $K$ defined in Section~\ref{Sect_2.0.4}.    

 For $L_c\in \cL(M_c)$ any $f$ in the space of invariants $\cH_L^S$ is the extension by zero under $H^S\hook{} H$. The space $\cH_L^S$ is naturally a $H_c$-module. We get an isomorphism of $H_c$-modules $\tau_{L_c}: \cH_{L_c}\,\iso\, \cH_L^S$ sending $f$ to the composition $H^S\toup{\alpha_S} H_c\toup{f} K$ extended by zero to $H$. 
 
  For $g\in G$, $L_c\in\cL(M_c)$ the isomorphism $g: \cH_L\,\iso\,\cH_{gL}$ yields an isomorphism $g: \cH_L^S\,\iso\,\cH_{gL}^S$ of $S$-invariants. 
 
\sssec{} Given $L_c^0, N_c^0\in \wt\cL(M_c)(\Fp)$, we define a canonical intertwining operator 
\begin{equation}
\label{operator_cF_can}
\cF_{N_c^0, L_c^0}: \cH_L\,\iso\,\cH_N
\end{equation}
as the unique isomorphism of $H$-modules such that the diagram commutes
$$
\begin{array}{ccc}
\cH_L^S & \;\toup{\cF_{N_c^0, L_c^0}}\; &\cH_N^S\\
\uparrow\lefteqn{\scriptstyle \tau_{L_c}} && \uparrow\lefteqn{\scriptstyle \tau_{N_c}} \\
\cH_{L_c} & \;\toup{F_{N_c^0, L_c^0}}\; & \cH_{N_c}
\end{array}
$$
Here $F_{N_c^0, L_c^0}$ are the canonical intertwining operators from Section~\ref{Sect_3.1.4}. The properties of the canonical intertwining operators of Section~\ref{Sect_3.1.4} imply the following propeties of (\ref{operator_cF_can}): 

\begin{itemize}
\item $\cF_{L_c^0, L_c^0}=\id$ for $L_c^0\in \wt\cL(M_c)(\Fp)$;
\item for $R_c^0, N_c^0, L_c^0\in \wt\cL(M_c)(\Fp)$ one has
$$
\cF_{R_c^0, N_c^0}\comp F_{N_c^0, L_c^0}=F_{R_c^0, L_c^0}
$$
\item for $g\in G, N_c^0, L_c^0\in \wt\cL(M_c)(\Fp)$ we have $g\comp F_{N_c^0, L_c^0}\comp g^{-1}=F_{gN^0, gL^0}$.
\end{itemize}

\begin{Def} Let $\pi$ be the $K$-vector space of collections $f_{L_c^0}\in \cH_L$ for $L_c^0\in \wt\cL(M_c)(\Fp)$ satisfying the property: for $N_c^0, L_c^0\in \wt\cL(M)(\Fp)$ one has
$$
\cF_{N_c^0, L_c^0}(f_{L_c^0})=f_{N_c^0}
$$
The element $h\in H$ sends $\{f_{L_c^0}\}\in\pi$ to the collection $\{h(f_{L_c^0})\}\in\pi$. This is our canonical $H$-representation over $K$.
\end{Def}

 The group $G$ acts on $\pi$ sending $\{f_{L_c^0}\}\in\pi$ to the collection $L_c^0\mapsto  g(f_{g^{-1}L_c^0})$. This is a version of the Weil representation of $G$. (In the case when the field of coefficients is $\CC$, this $G$-representation was also obtained in \cite{CMS}, however a canonical representation of $H$ was not constructed in \cite{CMS}). 
 
 The above actions of $H$ and $G$ on $\pi$ combine to an action of the semi-direct product $H\rtimes G$ on $\pi$. Theorem~\ref{Th_1} is proved.
 
\medskip\noindent
{\bf Acknowledgements.} I am grateful to Sam Raskin for an email discussion of the subject.


\begin{thebibliography}{99}
\bibitem{B} A. Beilinson, Langlands parameters for Heisenberg modules. In: Bernstein J., Hinich V., Melnikov A. (eds) Studies in Lie Theory. Progress in Mathematics, vol 243, Birkhauser Boston (2006)
\bibitem{CMS} G. Ciff, D. McNelly, F. Szechtman, Weil representations of symplectic groups over rings, Journal of the London Mathematical Society , Volume 62 , Issue 2 , October 2000 , 423 - 436
\bibitem{GL} D. Gaitsgory, S. Lysenko, Parameters and duality for the metaplectic geometric Langlands theory, Selecta Math., (2018) Vol. 24, Issue 1, 227-301, erratum available at \\
{\tt https://lysenko.perso.math.cnrs.fr/papers/twistings\_30aout2020.pdf}
\bibitem{Ko} M. Kosters, Classification on symmetric non-degenerate bilinear forms on finite Abelian groups, \\
{\tt http://www.math.leidenuniv.nl/$\sim$edix/tag\_2009/michiel\_1.pdf}
\bibitem{LL} V. Lafforgue, S. Lysenko,  Geometric Weil representation: local field case, Compos. Math. 145 (2009), no. 1, 56 - 88
\bibitem{L} S. Lysenko, Twisted geometric Langlands correspondence for a torus,  IMRN, 18, (2015),  8680 - 8723
\bibitem{P} D. Prasad, Notes on central extensions,\\
{\tt http://www.math.tifr.res.in/$\sim$dprasad/dp-lecture4.pdf}
\end{thebibliography}
\end{document}